\documentclass{hjm1}
      \usepackage{amsmath, amssymb}

       \def\diag{\mathop{\rm diag}\nolimits}
       \def\Ker{\mathop{\rm Ker}\nolimits}
       \def\Rank{\mathop{\rm Rank}\nolimits}
       \def\Spec{\mathop{\rm Spec}\nolimits}

\begin{document}

   \author{Dan Kucerovsky}
   \address{D. Kucerovsky, Department of Mathematics, University of New Brunswick, Fredericton, NB E3B 5A3, Canada}
   \email{dkucerov@unb.ca}

   \author{Aydin Sarraf}

   \address{A. Sarraf, Department of Mathematics, University of New Brunswick, Fredericton, NB E3B 5A3, Canada}

\email{Aydin.Sarraf@unb.ca}

   \title[Schur Maps]{Schur multipliers and matrix products}

   \begin{abstract}
      We give necessary and sufficient conditions for a Schur map to be a homomorphism, with some generalizations to the infinite-dimensional case. In the finite-dimensional case, we find that a Schur multiplier $[a_{ij}]\colon M_{n}(  \mathbb{C}  ) \rightarrow M_{n}(  \mathbb{C}  )$ distributes over matrix multiplication if and only if  $a_{ij}=f(i)/f(j)$ for some $f\colon {\mathbb N}\longrightarrow {\mathbb C^{*}},$
      where $\mathbb{C}^*$ is the set of nonzero complex numbers.  In addition, it is shown that it is possible to enumerate all $\ast$-preserving multiplicative Schur maps on $M_{n}(  \mathbb{R}  )$. We also study the relation of Schur map to the extreme points of certain sets.
   \end{abstract}


   \maketitle

   \section{Introduction}

Let $M_{n}(  \mathbb{C}  )$ be the algebra of all $n\times n$ complex matrices. The Schur product
$A\circ B$ of two matrices $A =(a_{ij})$ and $B =(b_{ij})$ in $M_{n}(  \mathbb{C}  )$ is defined by $A \circ B:=(a_{ij} b_{ij})$ and  given $A\in M_{n}(  \mathbb{C}  )$, the Schur map $S_{A}$ on $M_{n}(  \mathbb{C}  )$ is defined
by $S_{A}(B):=A \circ B$ for all $B\in M_{n}(  \mathbb{C}  )$. Although a Schur map certainly appears unlikely to be a homomorphism with respect to ordinary matrix multiplication, one can verify that, for example:$$  \left(\begin{smallmatrix}a&b\\c&d\\ \end{smallmatrix}\right)\mapsto \left(\begin{smallmatrix}a&ib\\-ic&d\\ \end{smallmatrix}\right)$$ is both a Schur map and a $*$-homomorphism. In this article, we develop several characterizations of multiplicative Schur maps. In the infinite-dimensional case, there is a natural notion of a Schur multiplier of a Schatten ideal of $B(H).$ If the Hilbert space is taken to be finite-dimensional, a Schur multiplier is a Schur map.

In our discussion, $J$ denotes the $n\times n$ matrix with all entries equal to $1$, and $I$ is the usual identity matrix. We use the notation $E_{11},E_{12}, . . . ,$ $E_{nn}$ for the standard basis of $M_{n}(  \mathbb{C}  )$ or their infinite--dimensional analogues, as  in section $4$. If A has no zero entries then by $A^{[-1]}$ we mean the matrix whose entries are reciprocals of the entries of $A$. By $\Spec(A)$ we mean the set of all eigenvalues of the matrix $A.$ The notation $A^{\ast}$ stands for the conjugate transpose of $A$. By a positive matrix, we always mean a Hermitian matrix with nonnegative spectrum. Throughout this paper, we generally exclude the zero Schur map from the set of Schur maps in order to avoid some trivial cases in the statements of our results. In section $2$, we characterize multiplicative Schur maps over $M_{n}(  \mathbb{C}  )$ and it turns out that in the case of positive Schur maps, the entries of the corresponding matrix are nothing but the extreme points of the closed unit disk, thus, are necessarily scalars of modulus 1. In section $3$, we enumerate all positive Schur maps over $M_{n}(  \mathbb{R}  )$. In section $4$, we try to generalize our results to the infinite--dimensional setting. Finally, in section $5$ we explore the relation of Schur maps to the extreme points of certain sets.

\section{Characterization in finite dimension}
A map $\Phi$ from $M_{n}(  \mathbb{C}  )$ to $M_{n}(  \mathbb{C}  )$ is called \textit{positive} if it maps the positive matrices in $M_{n}(  \mathbb{C}  )$ to  positive matrices in $M_{n}(  \mathbb{C}  )$.

\begin{thm}\label{main}
Let $S_{A}:  M_{n}(  \mathbb{C}  ) \longrightarrow M_{n}(  \mathbb{C}  )$ be a nonzero Schur map. Then the following are equivalent:

\begin{enumerate}
\item $S_{A}$ is multiplicative,
\item $S_{A}$ is spectrum preserving,
\item $\Spec(A)= \{0, n\}$ where $0$ has multiplicity $n-1$ and the diagonal entries of $A$ are equal to $1$,

\item $A$ is of rank one with diagonal entries equal to $1$, and

\item $a_{ij}$ = $a_{ik}$ $a_{kj}$  and $a_{ii}=1$ for all $1 \leq i,j,k\leq n$.
\end{enumerate}

\end{thm}

\begin{proof}
$(i)\Rightarrow(ii):$ If $S_{A}$ is multiplicative we claim that it is bijective and that the inverse is a Schur map. In other words, $A$ has a Schur inverse or equivalently $a_{ij}\neq0$  for all $1 \leq i,j\leq n$. To prove the claim, note that if there exist $i$ and $j$ such that $a_{ij}=0$ then  $S_{A}(E_{ij})=0,$ which implies that the kernel of the multiplicative map $S_A$ is nonempty. Since  $M_{n}(  \mathbb{C}  )$ is simple, this would imply that $S_A$ is in fact the zero map, which by hypothesis is not the case. Thus, the kernel is $\{0\}$, and $S_{A}$ is an automorphism, but any automorphism over $M_{n}(  \mathbb{C}  )$ is inner. In other words, there exists an invertible matrix $C$ such that $S_{A}(B)=CBC^{-1}$ for all $B\in M_{n}(  \mathbb{C}  )$. Thus,  $S_{A}$ is given by a similarity transformation, implying that the spectrum of $S_A (B)$ is equal to the spectrum of $B.$

 $(ii)\Rightarrow(iii)$: Since $S_{A}$ is spectrum preserving,  we can conclude that $\Spec(A)=\Spec(S_{A}(J))=\Spec(J).$ The spectrum of $J,$ hence of $A,$ is  $\{0, n\}.$ The kernel of  $S_{A}$ is $\{0\}$ because $S_A$ is a linear and spectrum-preserving map (see \cite[Prop. 2.1.i]{A}.) In finite dimensions, injectivity implies surjectivity, so that $S_A$ is hence surjective.  Being spectrum-preserving, $S_{A}$ preserves the trace, but the equation ${\rm Tr}(S_A(B))={\rm Tr}(B)$ means that $\sum a_{ii} b_{ii} = \sum b_{ii}$ for all choices of sequence $(b_{ii}),$ and this can only occur if $a_{ii}=1$ for all $1\leq i\leq n$. Hence, ${\rm Tr}(A)=n$ and since the spectrum of $A$ is $\{0,n\}$ we deduce that the algebraic multiplicities of $0$ and $n$ in the spectrum are respectively $n-1$ and $1$.

$(iii)\Rightarrow(iv):$ Since rank is the number of nonzero eigenvalues, with multiplicity, we have $\Rank(A)=1$.

$(iv)\Rightarrow(v):$ Since $\Rank(A)=1$, we conclude that the $j^{th}$ column of $A$ is a scalar multiple of the $k^{th}$ column, that is, $a_{ij}=\lambda_{jk}a_{ik}$ but for $i=k$ we have  $ a_{kj}=\lambda_{jk}a_{kk}=\lambda_{jk}$. Therefore, $a_{ij}$ = $a_{ik}$ $a_{kj}$  and $a_{ii}=1$ for all $1 \leq i,j,k\leq n$.

$(v)\Rightarrow(i):$ For arbitrary matrices $B$ and $C$, it follows that $$(S_{A}(B)S_{A}(C))_{ij}$$ $$ =\displaystyle\sum\limits_{k=1}^n a_{ik}b_{ik}c_{kj}a_{kj}=\displaystyle\sum\limits_{k=1}^n a_{ik}a_{kj}b_{ik}c_{kj}=\displaystyle\sum\limits_{k=1}^n a_{ij}b_{ik}c_{kj}$$ $$=(S_{A}(BC))_{ij}$$ for all $1\leq i,j \leq n$.
\end{proof}
\begin{rmk} Matrices of positive elements satisfying the condition $a_{ij}$ = $a_{ik}$ $a_{kj}$ that appears in condition $v$ of Theorem \ref{main} are in fact the so-called consistent matrices that are used in one branch of mathematical
economics (see \cite[chapter 7]{Saaty}.)  The requirement of having unit elements on the diagonal is then equivalent to the condition that the matrix is what is called a reciprocal matrix. One could say that Theorem \ref{main} characterizes multiplicative Schur maps in terms of a generalization of the  consistent reciprocal matrices.
\end{rmk}
\begin{cor} An unital Schur map $S_A\colon M_n \longrightarrow M_n$ is a homomorphism with respect to the usual matrix product if and only if the coefficients of $A$ are of the form
$a_{ij}=f(i)/f(j)$ where $f$ is a vector of nonzero complex numbers. Moreover, if $S_A$ is a homomorphism,
it is constructed as a similarity transformation by a diagonal matrix.\label{cor:diag}\end{cor}
\begin{proof}
If $S_A$ is a homomorphism, then by the above Theorem $a_{ij}$ = $a_{ik}$ $a_{kj}$  where $a_{ii}=1$ and all the $a_{ij}$ are nonzero, so that $$a_{ik}=\frac{a_{i1}}{a_{k1}}.$$
We can set $f(i):=a_{i1}.$ Conversely, if  $a_{ij}=f(i)/f(j)$ then it follows that for any matrix $B$:
$$S_A (B)=\left(\begin{smallmatrix}f(1)&&\\&\ddots&\\&&f(n)\\\end{smallmatrix}\right)
                                            B
\left(\begin{smallmatrix}\frac{1}{f(1)}&&\\&\ddots&\\&&\frac{1}{f(n)}\\\end{smallmatrix}\right) $$
where the operation on the right is ordinary matrix multiplication. But, then $S_A$ is a homomorphism with respect to matrix multiplication.
\end{proof}

\begin{cor}\label{norm}
If $S_{A}:  M_{n}(  \mathbb{C}  ) \longrightarrow M_{n}(  \mathbb{C}  )$ is multiplicative then the following hold:
\begin{enumerate}
\item $A^{\ast}=\overline{A^{[-1]}}$ and $A$ is diagonalizable.

\item In general, $\|A\|\geq n,$ and if $A^{*}=A$ then $\|A\|=n$.

\item If $A=A^*$ then $S_{A}$ is numerical range preserving.

\end{enumerate}

\end{cor}

\begin{proof}

$(i)$ By Theorem \eqref{main}, $1=a_{ii}=a_{ij}a_{ji}$ which implies $a_{ij}^{\ast}=\overline{a_{ji}}=\dfrac{1}{\overline{a_{ij}}}$ for all $1 \leq i,j\leq n$. Since the  polynomial $p(x)=x(x-n)$ has distinct roots and annihilates $A$, it follows that the minimal polynomial of $A$ has distinct roots, and thus $A$ is diagonalizable.

$(ii)$ Since $S_A$ is multiplicative, Theorem \ref{main} shows that the spectrum of $A$ is $\{0,n\},$ so the spectral radius $\rho(A)$ of $A$ is $n.$ Therefore $\|A\|\geq\rho(A)=n.$ If $A=A^{\ast}$, then $\|A\|=\rho(A)=n$.


$(iii)$ As in the proof of Corollary \ref{cor:diag}, we have that $S_A(x)=\Lambda x \Lambda^{-1},$ where the diagonal entries of the diagonal matrix $\Lambda$ come from the first row of $A.$ But by part $i$ and the fact that $A^*=A$ we conclude that $\overline{a_{ij}}=1/a_{ij}$. This property carries over to the entries of $\Lambda$ so that $\overline{\Lambda}=\Lambda^{-1}.$ But then $\Lambda^*=\Lambda^{-1}$, making $\Lambda$ a unitary. We thus have $S_{A}(x)=\Lambda x \Lambda^*$ as claimed. By \cite{P} it follows that $S_{A}$ is numerical range preserving.

\end{proof}
\begin{thm}\label{mainp}

Let $S_{A}:  M_{n}(  \mathbb{C}  )\longrightarrow M_{n}(  \mathbb{C}  )$ be an unital Schur map. Then the following are equivalent:
\begin{enumerate}
\item $S_{A}$  is multiplicative and $\ast$-preserving,

\item $S_{A}$  is a completely positive isomorphism,

\item $A$ is of rank one and normal with diagonal entries equal to $1$,

\item $A$ is of rank one with unimodular entries and diagonal entries equal to $1$, and

\item $A$ is self-adjoint with $\Spec(A)= \{0, n\}$ where $0$ has multiplicity $n-1$ and $\|S_{A}\|=1$,

\item $A$ and $A^{[-1]}$ are positive matrices with diagonal entries equal to $1$.
\end{enumerate}

\end{thm}
\begin{proof}

$(i)\Rightarrow(ii):$ Since $S_{A}$ is $\ast$-preserving, $A$ is self-adjoint and by Theorem \eqref{main} it has positive spectrum. Therefore, $A$ is positive, and by \cite[Theorem 3.7]{VP} the map $S_A$ is completely positive. On the other hand, since by Theorem \eqref{main} the map $S_{A}$ has no kernel it is, by finite dimensionality, an invertible operator. Positive, unital, multiplicative, and invertible maps of matrix algebras have inverses that are positive, unital, and multiplicative. Clearly $S_{A^{[-1]}}$ is an inverse for $S_A$. Arguing as previously, $S_{A^{[-1]}}$ is completely positive. Therefore, $S_{A}$ is a completely positive isomorphism.\vspace{.4cm}

$(ii)\Rightarrow(iii):$ If $(ii)$ holds then $S_A$ is a $*$-isomorphism of matrix algebras. The map $S_A$ thus maps a projection of rank $k$ to a projection of rank $k$. Applying $S_A$ to the orthogonal projection of rank 1 given by $\frac{1}{n}J,$ we find that $\frac{1}{n}A$ is a projection of rank 1, and is positive. Thus $A$ is positive, hence normal, and of rank 1 as a matrix. The matrix $A$ has diagonal elements $1$ since $S_A$ is unital.
\vspace{.4cm}

$(iii)\Rightarrow(iv):$ Since $A$ is normal, it has a decomposition of the form $U \Lambda U^{\ast}$ where $U$ is a unitary matrix and $\Lambda=\diag(\lambda_{1},\lambda_{2},...,\lambda_{n})$ where $\lambda_{i}$ are the eigenvalues of $A$ with the decreasing order. Since $A$ has rank one with diagonal entries equal to $1$ we can conclude by Theorem \eqref{main} that $\Lambda=\diag(n,0,...,0)$.
 Therefore, $a_{ij}=\displaystyle\sum\limits_{k=1}^n \displaystyle\sum\limits_{l=1}^n u_{il}\gamma_{lk}\overline{u_{jk}}$ but $\gamma_{lk}=0$ unless $l=k=1$ where $\gamma_{11}=\lambda_{1}=n$. Hence, $a_{ij}=nu_{i1}\overline{u_{j1}}$ which implies $1=a_{ii}=n\left\lvert u_{i1} \right\rvert^{2}$ for all $1 \leq i\leq n$. Thus, $\left\lvert a_{ij} \right\rvert^{2}=a_{ij}\overline{a_{ij}}=n^{2}\left\lvert u_{i1} \right\rvert^{2}\left\lvert u_{j1}\right\rvert^{2}=1$ for all $1 \leq i,j\leq n$.\vspace{.4cm}

$(iv)\Rightarrow(v):$ Since $A$ is of rank one with diagonal entries equal to $1$, Theorem \eqref{main} implies that $\Spec(A)= \{0, n\}$ where $0$ has multiplicity $n-1$. If $A$ has unimodular entries then $\displaystyle\sum\limits_{k=1}^n (a_{ik}\overline{a_{jk}}-a_{kj}\overline{a_{ki}})=\displaystyle\sum\limits_{k=1}^n(\frac{a_{ik}}{a_{jk}}-\frac{a_{kj}}{a_{ki}})$ for all $1 \leq i,j\leq n$. By repeated use of Theorem \eqref{main} we can conclude that

 $\displaystyle\sum\limits_{k=1}^n(\frac{a_{ik}}{a_{jk}}-\frac{a_{kj}}{a_{ki}})=\displaystyle\sum\limits_{k=1}^n\frac{a_{ik}a_{ki}-a_{kj}a_{jk}}{a_{jk}a_{ki}}=
 \displaystyle\sum\limits_{k=1}^n\frac{a_{ii}-a_{jj}}{a_{ji}}= \displaystyle\sum\limits_{k=1}^n\frac{1-1}{a_{ji}}=0$ for all $1 \leq i,j\leq n$.
 Hence, $(AA^{\ast})_{ij}=\displaystyle\sum\limits_{k=1}^na_{ik}\overline{a_{jk}}=\displaystyle\sum\limits_{k=1}^na_{kj}\overline{a_{ki}}=(A^{\ast}A)_{ij}$ for all $1 \leq i,j\leq n$. Thus, $A$ is normal with positive spectrum and the spectral theorem for normal matrices implies that $A$ is positive and it is known that for positive matrices $\| S_{A}\|=\|S_{A}\|_{cb}=max\{a_{ii}: i=1,...,n \}=1.$\vspace{.4cm}

$(v)\Rightarrow(vi):$ It is clear that $A$ is positive. Therefore, by Corollary \eqref{norm} we conclude that $A=A^{\ast}=\overline{A^{[-1]}}$ which implies $\Spec(A^{[-1]})=\Spec(A)=\{0,n\}$. Thus, $A^{[-1]}$ is also positive. Since $A$ is positive, $a_{ii} \leq\|S_{A}\|=1$. On the other hand, ${\rm Tr}(A) =n$. Therefore, the diagonal entries of $A$ are equal to $1$.\vspace{.4cm}

$(vi)\Rightarrow(i):$ If $A$ and $A^{[-1]}$ are positive with diagonal entries equal to $1$ then by \cite[Theorem 3.7]{VP} we can conclude that $S_{A}$ and $S_{A^{[-1]}}$ are completely positive and unital. Denoting these maps by $\phi$ and $\phi^{-1}$ respectively, by the Schwarz inequality \cite[Corollary 2.8]{C} we have $ \phi(a^*)\phi(a)\leq\phi(a^*a)$. By using the Schwarz inequality once more for $\phi^{-1}$ we have $ a^*a=\phi^{-1}(\phi(a^*))\phi^{-1}(\phi(a))\leq\phi^{-1}(\phi(a^*)\phi(a))$. Hence, $ \phi(a^*)\phi(a)-\phi(a^*a)=\phi(\phi^{-1}(\phi(a^*)\phi(a))-a^*a)\geq0$. Therefore, $ \phi(a^*)\phi(a)=\phi(a^*a)$.

Let us define $\phi^{(2)}\colon M_2 \otimes M_n \longrightarrow M_2 \otimes M_n $ by $\mbox{Id}\otimes \phi.$  If we identify $ M_2 \otimes M_n$ with 2-by-2 matrices of operators (block matrices) we can write
$$\phi^{(2)}\left(
\begin{array}{cc}
a & b \\
c & d \\
\end{array}
\right) =
\left(\begin{array}{cc}
\phi(a) & \phi(b) \\
\phi(c) & \phi(d) \\
\end{array}
\right).$$
 Now, if we apply the Schwarz inequality to the positive map $\phi^{(2)}$ and the element

\[
\left(
\begin{array}{cc}
a & b^{*} \\
0 & 0 \\
\end{array}
\right)\]

  we get

  $$  \left(
\begin{array}{cc}
\phi(a) & \phi(b^{*}) \\
0 & 0 \\
\end{array}
\right)^*
 \left(
\begin{array}{cc}
\phi(a) & \phi(b^{*}) \\
0 & 0 \\
\end{array}
\right)
            \leq
 \left(
\begin{array}{cc}
\phi(a^*a) & \phi(a^*b^{*}) \\
\phi(ba) & \phi(bb^*) \\
\end{array}
\right).
 $$
  Thus we conclude that the  following  matrix is operator positive:

  \[
\left(
\begin{array}{cc}
 \phi(a^*a)-\phi(a^*)\phi(a) & \phi(a^*b^*)-\phi(a^*)\phi(b^*) \\
\phi(ba)-\phi(b)\phi(a) & \phi(bb^*)-\phi(b)\phi(b^*) \\
\end{array}
\right).\]
But we have already shown using the Schwartz inequality that the entries on the diagonal of the above matrix happen to be zero, so then we have that
  \[\left(
\begin{array}{cc}
 0 & \phi(a^*b^*)-\phi(a^*)\phi(b^*) \\
\phi(ba)-\phi(b)\phi(a) & 0 \\
\end{array}
\right)\] is operator positive.
 Let $c= \phi(ba)-\phi(b)\phi(a)$ and note that an operator matrix of the form
   \[\left(
\begin{array}{cc}
 \lambda{\rm Id} & c^* \\
c & \lambda{\rm Id} \\
\end{array}
\right),\] where $\lambda$ is a non-negative scalar, is operator positive if and only if $\|c\|\leq\lambda$. This can be deduced from, for example, \cite[Lemma 3.1]{VP}. But, then it follows that $\phi(ba)-\phi(b)\phi(a)$ is zero.
\end{proof}

\begin{prop} \label{idempotent}Let $S_A\colon M_n({\mathbb C})\longrightarrow M_n({\mathbb C})$ be a multiplicative Schur map. The following are equivalent:
\begin{enumerate}
	\item The map is $*$-preserving,
	\item The matrix $A$ has norm  $\|A\|=n,$
	\item The matrix $\frac{1}{n} A$ is an orthogonal projection, and
	\item The operator norm of $S_A$ is 1.
\end{enumerate}

\end{prop}
\begin{proof} $(i)\Rightarrow(ii):$ If $S_A$ is $*$-preserving then $A$ is self-adjoint and by Corollary \eqref{norm}, we have $\|A\|=n.$

$(ii)\Rightarrow(iii):$ Recall that by applying $S_A$ to the equation $J^2=nJ,$ we find that $p:=\frac{1}{n}A$ is a not necessarily orthogonal projection. By \cite[Theorem 5.4]{zhang}, a projection  $p$ is an orthogonal projection if and only if it has operator norm $\|p\|=1.$ Thus, the norm condition $\|A\|=n$ holds if and only if $p,$ and hence $A,$ is positive.

$(iii)\Rightarrow(i):$ If $\frac{1}{n}A$ is a projection then $A$ is self-adjoint and consequently $S_A$ is $*$-preserving.

For the equivalence of $i$ and $iv,$ we compute  $\| S_A\|$ with respect to the norm induced by the operator norm on the range and domain. Since $S_A$ is multiplicative, we note that $S_A (B)=\Lambda B \Lambda^{-1}$ by Corollary 2.2, and we have $\|S_A(B)\|\leq \|\Lambda \| \|B\| \|\Lambda^{-1}\|.$ This bound is in fact achieved if we take $B$ equal to a suitable matrix unit.
Thus, $$\|S_A\|=\|\Lambda \|  \|\Lambda^{-1}\|=\displaystyle\sup_{i,j} |\lambda_i/\lambda_j|,$$ and this is equal to 1 if and only if the $\lambda_i$ all have modulus 1. But by Theorem 2.4 this occurs if and only if the map $S_A$ is $\ast$-preserving.
\end{proof}

\section{Enumeration of matrices of multiplicative Schur maps}
  In this section, $\mathbb{F}$ denotes a field of characteristic zero. Although matrices can very rarely be diagonalized over a general field, we do still have that
  $M_{n}(  \mathbb{F}  ) $ is simple and that  all automorphisms are inner (as a consequence of the Noether-Skolem theorem). Moreover, it is still true that the matrix units $E_{ij}$ satisfy $E_{ij} = E_{ik} E_{kj},$ and from these facts it follows that:

  \begin{prop} Let $\mathbb F$ be a field. Let $S_A\colon  M_{n}(  \mathbb{F}  ) \longrightarrow M_{n}(  \mathbb{F}  )$ be a Schur map.  Then the following are equivalent:
  \begin{enumerate}\item $S_A$ is a nonzero homomorphism,
  \item $S_A (x) = \Lambda x \Lambda^{-1},$ for some invertible diagonal matrix $\Lambda,$ and
  \item $a_{ij}$ = $a_{ik}$ $a_{kj}$  and $a_{ii}=1$ for all $1 \leq i,j,k\leq n$.
  \end{enumerate}
  \end{prop}
   We define generalized Toeplitz matrices by $G^{n}_{\mathbb{F}}:=\{A\in M_{n}(  \mathbb{F}  )| a_{ij}=\lambda^{j-i}, \lambda\in\mathbb{F}^{*}\}$. We define $L^{n}_{\mathbb{F}}$ to be  the set of all $A\in M_{n}(\mathbb{F})$ such that $S_{A}:  M_{n}(  \mathbb{F}  )  \longrightarrow M_{n}(  \mathbb{F}  )$ is multiplicative.
\begin{prop}

 If we equip $L^{n}_{\mathbb{F}}$ with the Schur product then $L^{n}_{\mathbb{F}}$ turns into an abelian group and $G^{n}_{\mathbb{F}}$ is a subgroup, consisting of Toeplitz matrices.

\end{prop}

\begin{proof}

It is clear that $L^{n}_{\mathbb{F}}$ is an abelian group with respect to Schur product because $J\in L^{n}_{\mathbb{F}}$ and if $A\in L^{n}_{\mathbb{F}}$ then $A^{[-1]}\in L^{n}_{\mathbb{F}}$ because the inverse of a homomorphism is also a homomorphism. Since $a_{ij}=\lambda^{j-i}=\lambda^{(j-1)-(i-1)}=a_{i-1j-1}$, $G^{n}_{\mathbb{F}}$ consist of Toeplitz matrices and the relations $a_{ij}=\lambda^{j-i}=\lambda^{k-i}\lambda^{j-k}=a_{ik}a_{kj}$ imply that it is a subgroup of $L^{n}_{\mathbb{F}}$.

\end{proof}

We remark that for $n=2$, we have $G^{2}_{\mathbb{F}}=L^{2}_{\mathbb{F}}$.

In general, returning to the case of the real or the complex field, the problem of determining all the coefficients of the matrix of a multiplicative Schur map given partial information on the coefficients leads to a Gauss-Jordan problem in abstract linear algebra. Transforming the equations of Theorem \eqref{main} by $x\mapsto \frac{\ln x}{2\pi i},$ we obtain the following linear system
\begin{align*} b_{ii}&=0\\ b_{ij}&= b_{ik}+b_{kj}
\end{align*} over $\mathbb{C}/\mathbb{Z}=\mathbb{C}/\{x\sim y | x-y\in \mathbb{Z} \}\cong S^1\times \mathbb{R}.$

If we want the Schur map to be multiplicative and $*$-preserving, then, by Theorem \eqref{mainp}, the equations for the transformed coefficients become:
\begin{align*} b_{ii}&=0\\ b_{ij}&= b_{ji} = b_{ik}+b_{kj}
\end{align*} over $\mathbb{R}/\mathbb{Z}=\mathbb{R}/\{x\sim y | x-y\in \mathbb{Z} \}\cong S^1.$ In general we are free to specify $n-1$ of the unknowns, and in the case where we specify one row or column of the matrix $[b_{ij}]$, it is straightforward to solve for the remaining unknowns. Let $\mathbb{T}$ be the circle group, then it is easy to check that the map $\Phi$ from $\mathbb{T}^{n-1}$ to  $(L^n_{\mathbb C})^{+}$ which sends the vector $(z_{1},...,z_{n-1})$ to the first row of the matrix $A$ (note that the first entry of the matrix is always $1$) is in fact an isomorphism between topological groups.
From this discussion, denoting by $(L^n_{\mathbb C})^{+}$ the positive matrices within $L^n_{\mathbb C}$ :
\begin{prop} The Lie group $L^n_{\mathbb C}$ has complex dimension $n-1, $ and the compact Lie group  $(L^n_{\mathbb C})^{+}$ is isomorphic to the  $(n-1)-$torus.
\end{prop}

If we are interested in Schur maps that are multiplicative, $*$-preserving, and have only real entries in their matrix of coefficients, then all the entries must, by  Theorem \eqref{mainp}, be in $\{-1,+1\}$. Then, for example, for $A\in(L^{3}_{\mathbb{R}})^{+},$ where $(L^{k}_{\mathbb{R}})^{+}$ is the subgroup of positive matrices within  $L^{k}_{\mathbb{R}},$
the only possibilities are:
\begin{center}\begin{equation*}
\left( \begin{array}{ccc}
1 &  1   &1\\
1 &  1   &1\\
1 &  1   &1\end{array} \right)
,\left( \begin{array}{ccc}
  1 &-1  &-1\\
-1 &  1   &1\\
-1 &  1   &1\end{array} \right)
,\left( \begin{array}{ccc}
  1 &-1 &  1\\
-1 &  1 & -1\\
  1 &-1 &  1\end{array} \right)
 {\rm and}\left( \begin{array}{ccc}
  1 &  1 &-1\\
  1 &  1 &-1\\
-1 & -1 &  1\end{array} \right).
\end{equation*}\end{center}\vspace{.2cm}

These are all of the 3-by-3 positive and real Schur multiplicative matrices. For the 2-by-2 complex case the only possibility is:
\begin{equation*}
\left( \begin{array}{cc}
1 &  z   \\
1/z &  1  \end{array} \right)
\end{equation*}\vspace{.2cm}
The above matrix has operator norm $\sqrt { {|z|^{-2}+2+|z|^{2}}},$ which is consistent with Proposition \eqref{idempotent}.
For the 3-by-3 complex case, we have a nice expression for the general form of a matrix for a multiplicative Schur map, as a Schur factorization:
\begin{equation*}
\left( \begin{array}{ccc}
1 &  z & z   \\
1/z &  1 & 1 \\
1/z & 1 & 1 \end{array} \right)\circ
\left( \begin{array}{ccc}
1 &  1 & w   \\
1 &  1 & w \\
1/w & 1/w & 1 \end{array} \right)
\end{equation*}\vspace{.2cm} where $\circ$ denotes the Schur product.

\begin{thm}
The discrete group
 $(L^{n}_{\mathbb{R}})^{+}$ has cardinality $2^{n-1}$.

\end{thm}

\begin{proof}

If $A\in (L^{n}_{\mathbb{R}})^{+}$ then by Theorem \eqref{mainp} we have either $a_{ij}=1$ or $-1$ for all $1 \leq i,j\leq n$. On the other hand, since $A$ is of rank one and $a_{11}$=$1$, $A$ can be completely determined by exactly $n-1$ of its first row's entries. In other words, we have $n-1$ entries and two choices for each entry. Each choice gives us a distinct Schur map, and therefore, $Card((L^{n}_{\mathbb{R}})^{+})=2^{n-1}$.

\end{proof}
 \begin{rmk} If we consider only Schur matrices with positive real entries, then the only such Schur map that is multiplicative and $\ast$-preserving is the identity map $S_{J}=\rm Id$.\end{rmk}

\section{Characterization in infinite dimension}
In this section, $H$ denotes the separable infinite--dimensional Hilbert space and $K(H)$ denotes the algebra of compact operators over $H$. If we fix, once and for all, an orthonormal basis for our Hilbert space, we can then identify bounded operators on $H,$ denoted by $B(H),$ with their (infinite) matrix representation with respect to that basis. We may then consider Schur multipliers with respect to various classes of linear operators in $B(H).$ The Schatten--von Neumann classes, ${\mathcal L}_p$, are algebraic ideals within $B(H)$ and are defined by $\{x\in B(H) |\, \tau(|x|^p)<\infty\},$ where $\tau$ is the canonical trace on $B(H).$ The Schatten $p$-norm, which makes ${\mathcal L}_p$ a Banach space within $K(H)$, is defined by $\|x\|:= \tau(|x|^p)^{1/p}.$
In particular, the Schur multipliers  of ${\mathcal L}_2$ are precisely the matrices whose coefficients are uniformly bounded.

Since the Schatten classes are ideals, we may take products of elements of a Schatten class ${\mathcal L}_p$ and obtain elements of ${\mathcal L}_p$. (For ${\mathcal L}_{2},$ more is true: a product of elements in ${\mathcal L}_{2}$ is also in ${\mathcal L}_1.$)
\begin{prop} A Schur multiplier $S_A\colon {\mathcal L}_{2} \longrightarrow {\mathcal L}_{2}$ with no kernel satisfies $S_A(xy)=S_A(x)S_A(y)$ if and only if $a_{ij}=f(i)/f(j),$ where $f\colon {\mathbb N}\longrightarrow {\mathbb C}$ is a bounded sequence that is bounded away from zero.
\end{prop}
\begin{proof} If $a_{ij}=f(i)/f(j),$ then given two operators $x,y\in {\mathcal L}_{2}$ that are of finite rank with respect to the chosen basis, we choose a matrix subalgebra $M_{n}(  \mathbb{C}  )$  within $K(H)$ that is a corner of our space of infinite matrices, and contains both $x$ and $y.$ We then have that $S_A(xy)=S_A(x)S_A(y)$ for $x$ and $y$ of finite rank. But since $S_A:{\mathcal L}_{2}\longrightarrow {\mathcal L}_{2}$ is a bounded operator, we then have by approximation that the equation $S_A(xy)=S_A(x)S_A(y)$ holds for all $x$ and $y$ in ${\mathcal L}_{2}$.

For the converse, if $S_A:{\mathcal L}_{2}\longrightarrow {\mathcal L}_{2}$ is multiplicative, then for every finite-dimensional matrix corner $M_{n}(  \mathbb{C}  )$  the map $S_A$ restricts to $S_A:M_{n}(  \mathbb{C}  ) \longrightarrow M_{n}(  \mathbb{C}  ) .$ The restricted maps are homomorphisms with no kernel, and thus by finite-dimensionality of $M_{n}(  \mathbb{C}  )$  they are isomorphisms. In particular, then, since the unit of an algebra is unique, the restricted maps $S_A$ must take the unit of $M_{n}(  \mathbb{C}  )$  to the unit of $M_{n}(  \mathbb{C}  ).$ Our results for the finite-dimensional case then imply that $a_{ij}=f(i)/f(j)$ with $f(i):=a_{i1}.$ Taking larger and larger corners, we conclude that for all $i$ and $j$ we have  $a_{ij}=f(i)/f(j)$ with $f(i):=a_{i1}.$ If either $f$ were not a bounded sequence, or $1/f$ were not a bounded sequence, then we could find an unbounded sequence of coefficients of $A,$ and $S_A$ would not be a Schur multiplier of ${\mathcal L}_{2}.$\end{proof}
We next consider the case of Schur multipliers of the compact operators.
\begin{lem}
Let $A$ be an infinite matrix in $\ell^{\infty} (\mathbb{N}^{2}) $ and $K(H)$ be the algebra of compact operators on $\ell^{2}(\mathbb{N})$ represented by formal infinite matrices. Then $S_{A}$ is a bounded operator from $K(H)$ to $K(H)$.
\end{lem}

\begin{proof}
It is sufficient to consider the case of elementary compact operators $T=e_{i}<e_{j}, .>$. Since $A\circ T=a_{ij}T$, we have $\|A\circ T\|\leq\| A\|_{\infty} \|T\|$ where $\| A\|_{\infty}=\sup_{i,j}\left\lvert a_{ij} \right\rvert$. Hence, $S_{A}$ is continuous and since it maps finite rank operators to finite rank operators, it is a map from $K(H)$ to $K(H)$.
\end{proof}

 \begin{thm}
Let $S_{A}:K(H)\longrightarrow K(H)$ be a Schur map where A is in $\ell^{\infty} (\mathbb{N}^{2}) $. Then $S_{A}$ is multiplicative if and only if the columns of $A$ are scalar multiples of the first column, and moreover there are 1's on the diagonal.

 \end{thm}

\begin{proof}
If $S_{A}$ is multiplicative then  $a_{ij}\neq 0$ for all $1 \leq i,j\leq n$ because if there exist a $m$ and $n$ in $\mathbb{N}$ such that $a_{mn}$=$0$ then $S_{A}(E_{mn})$=$0$ which implies $E_{mn}\in \Ker (S_{A})\lhd K(H)$ but this contradicts the fact that $K(H)$ is simple. Since $S_{A}$ is multiplicative we have $S_{A}(E_{ij})$=$S_{A}(E_{ik}E_{kj})$= $S_{A}(E_{ik})S_{A}(E_{kj})$ which implies $a_{ij}$=$a_{ik}a_{kj}$. If we let $k=i=j$ we have $a_{ii}=a_{ii}^{2}$ and therefore $a_{ii}=1$. For $i=1$, the above equations imply that
\begin{equation*}
\left(
\begin{array}{c}
 a_{1j}\\
 a_{2j}\\
 a_{3j}\\
 a_{4j} \\
  .\\
  .\\
  .\\
\end{array}\right)
=a_{kj}\left( \begin{array}{ccc}
 a_{1k}\\
 a_{2k}\\
 a_{3k}\\
 a_{4k} \\
  .\\
  .\\
  .\\
\end{array}\right)
\end{equation*}

If we let $k=1$ then $A$=[ $C_{1}$,  $a_{12}  C_{1}$,  $a_{13} C_{1}$,  $a_{14} C_{1}$, ...]. In other words, by abuse of terminology, we have $\Rank(A)=1$. Conversely if $\Rank(A)=1$ in this sense then we have
\begin{equation*}
\left(
\begin{array}{c}
 a_{1j}\\
 a_{2j}\\
 a_{3j}\\
 a_{4j} \\
  .\\
  .\\
  .\\
\end{array}\right)
=\lambda_{kj}\left( \begin{array}{ccc}
 a_{1k}\\
 a_{2k}\\
 a_{3k}\\
 a_{4k} \\
  .\\
  .\\
  .\\
\end{array}\right)
\end{equation*}
Since A is unit diagonal, if we let $i=k$ then we have  $a_{kj}= \lambda_{kj} a_{kk}= \lambda_{kj}$ and consequently  we recover the equations  $a_{ij}=a_{ik} a_{kj}$. Hence, $(A\circ BC)_{ij}=\displaystyle\sum\limits_{k=1}^{\infty} a_{ij} b_{ik} c_{kj}=\displaystyle\sum\limits_{k=1}^{\infty}a_{ik} b_{ik} c_{kj}a_{kj}=( (A\circ B)(A\circ C))_{ij}$. Therefore, $S_{A}(BC)=S_{A}(B)S_{A}(C)$.
\end{proof}

\begin{cor}

Let $S_{A}:K(H)\longrightarrow K(H)$ be $\ast$-preserving where A is in $\ell^{\infty} (\mathbb{N}^{2}) $. If $S_{A}$ is multiplicative then $\left\rvert a_{ij} \right\rvert$=$1$ and $a_{ii}$=$1$.

\end{cor}

\begin{proof}
By the above theorem, $A$ satisfies the equations $1$=$a_{ii}$=$a_{ij}a_{ji}=a_{ij}\overline{a_{ij}}=\left\rvert a_{ij} \right\rvert^{2}$. Therefore, $A$ has unimodular entries.
\end{proof}
\begin{prop}
Let $A$ be an infinite matrix in $ \ell^{\infty} (\mathbb{N}^{2})$. If $A$ has  1's on the diagonal and linearly dependent columns then A is not bounded in the operator norm.
\end{prop}
\begin{proof}

Choosing an arbitrary $n\in\mathbb{N}$, consider the $n$-by-$n$ top right corner of $A$ and denote it by $A_{n}$. By Theorem \eqref{main}, we can find an eigenvector $X_n$ such that $A_{n}X_n=n X_n$. Construct the vector $X^{\prime}$ such that $x^{\prime}_{i1}=x_{i1}$ for all $1\leq i \leq n$, with the other entries set to zero. We can multiply $X'$ by a scalar such that the norm of $X'$ is 1. But, $\|AX'\|\geq n.$ Thus, we have a sequence of elements of the unit ball upon which $A$ is unbounded, or in other words, $A$ is not bounded in the operator norm.

\end{proof}

\section{Relation to extreme points}


A linear completely positive map between von Neumann algebras is called normal if it is continuous with respect to the $\sigma$-weak topologies on domain and range. A Markov map is a linear normal trace preserving unital completely positive map from $M_{n}(  \mathbb{C}  )$ to $M_{n}(  \mathbb{C}  )$ and we use the notation $D(M_{n}(  \mathbb{C}  ))$ for the set of all Markov maps.
A positive matrix with unit diagonal entries is called a correlation matrix and the set of all correlation matrices is denoted with $\xi_{n}(  \mathbb{C}  )$.

\begin{prop}
\cite[Lemma 2.4]{BPS} Let $A\in M_{n}(  \mathbb{C}  )$ and $S_{A}:  M_{n}(  \mathbb{C}  ) \longrightarrow M_{n}(  \mathbb{C}  )$ be a Schur map. Then $A$ is an extreme point of $\xi_{n}(  \mathbb{C}  )$ if and only if $S_{A}$ is an extreme point of $D(M_{n}(  \mathbb{C}  ))$.
\end{prop}

\begin{rmk}

If $S_{A}:  M_{n}(  \mathbb{C}  ) \longrightarrow M_{n}(  \mathbb{C}  )$ is $\ast$-preserving and multiplicative then by Theorem \eqref{mainp}, the entries of $A$ are the extreme points of the closed unit disk. In addition, $A$ is a rank one correlation matrix and consequently by \cite[Theorem 1]{LT} it is an extreme point of $\xi_{n}(  \mathbb{C}  )$. However, this can also be proved by the above proposition. Note that by \cite[Theorem 3.5]{ES}, $S_{A}$ is an extreme point of the set of all positive unital Schur maps. In other words, it is an extreme Markov map and the above proposition implies that $A$ is an extreme correlation matrix.
\end{rmk}

\begin{prop}
Let $S_{A}:B(H)\longrightarrow B(H)$ be a Schur map with no kernel, where $A$ is the matrix of an operator in
the unit ball $B(H)_{1}$ of $B(H)$. If $S_{A}$ belongs to an extremal ray of the set of Schur maps, then $A$ is either a scalar multiple of an isometry or a scalar multiple of a co-isometry.
\end{prop}
\begin{proof}
By considering the unit ball, we may work with extreme points rather than extremal rays. Suppose $S_{A}$ is extreme and $ A=\lambda B + (1-\lambda) C$ where $\lambda \in (0,1)$ and $B,C$ are matrices of operators in the unit ball $B(H)_{1}$. By the properties of Schur maps, $S_{B}$ and $S_{C}$ are in the unit ball of the space of all Schur maps from $B(H)$ to $B(H$) and $S_{A}=\lambda S_{B} + (1-\lambda) S_{C}$. Since $S_{A}$ is extreme, $S_{B}=S_{C}$. Therefore, $S_{B-C}$ is the zero map. We claim that $B=C$ because if $B-C\neq0$ then it has a nonzero entry such as $d_{mn}$ and by applying $S_{B-C}$ to $E_{mn}$ we have $S_{B-C} (E_{mn} )\neq0$ which is a contradiction. Therefore $A$ is an extreme point of the unit ball of $B(H)$, and for $B(H)$ these are the isometries and co-isometries.
\end{proof}

\end{document}